%% file: DaiRigXia13-Update.tex
\documentclass[dvips,sts]{imsart}

\usepackage{amsthm,amsmath, amssymb}
\RequirePackage[colorlinks,citecolor=blue,urlcolor=blue]{hyperref}

\RequirePackage{hypernat}

\usepackage{color}
\usepackage{amsfonts, fancybox}
\usepackage{amssymb}
\usepackage[american]{babel}
\usepackage{graphicx, pstricks, pst-node, color}
\usepackage{psfrag}\usepackage{subfigure}
\usepackage{flafter, bm, dsfont}
\usepackage[section]{placeins}

\usepackage{multirow}
\usepackage{color}
\usepackage{amsmath}
\usepackage{float}
\usepackage{mathrsfs}
\usepackage{amsfonts}
\usepackage{dsfont}
\usepackage{amssymb}
\usepackage[boxed]{algorithm2e}
\usepackage{graphicx, amssymb, latexsym}
\graphicspath{{figures}}
\DeclareGraphicsExtensions{.eps.gz,.eps,.epsi.gz,.epsi,.ps,.ps.gz}
\usepackage{cite}
\usepackage{boxedminipage}

\newcommand{\bc}{\begin{center}}
\newcommand{\ec}{\end{center}}
\newcommand{\ba}{\begin{array}}
\newcommand{\ea}{\end{array}}
\newcommand{\be}{\begin{eqnarray}}
\newcommand{\ee}{\end{eqnarray}}
\newcommand{\bel}{\begin{eqnarray}\label}
\newcommand{\eel}{\end{eqnarray}}
\newcommand{\bes}{\begin{eqnarray*}}
\newcommand{\ees}{\end{eqnarray*}}
\newcommand{\bi}{\begin{itemize}}
\newcommand{\ei}{\end{itemize}}
\newcommand{\bn}{\begin{enumerate}}
\newcommand{\en}{\end{enumerate}}

\newcommand{\sB}{\mathsf{B}}
\newcommand{\dlots}{\ldots}

\newcommand{\blog}{\mathop{\overline{\log}}}

\newcommand{\iid}{{\it i.i.d.\ }}

\newcommand{\bphi}{\boldsymbol{\phi}}
\newcommand{\btheta}{\boldsymbol{\theta}}

\newcommand{\by}{\boldsymbol{y}}
\newcommand{\e}{\mathsf{e}}

\newtheorem{lemma}{Lemma}

\newtheorem{condition}{Condition}
\newtheorem{theorem}{Theorem}
\newtheorem{corollary}{Corollary}

\newcommand{\tr}{\mathop{\mathsf{Tr}}}
\newcommand{\rk}{\mathop{\mathsf{Rk}}}
\newcommand{\diag}{\mathop{\mathsf{diag}}}
\newcommand{\linspan}{\mathop{\mathsf{span}}}
\newcommand{\MSE}{\mathop{\mathsf{MSE}}}

\newcommand{\thetaphat}{\hat{\theta}_k}
\newcommand{\phihat}{\hat{\phi}}

\newtheorem{lemmapp}{Lemma}[section]

\input{commandes}

{
\begin{center}
\begin{boxedminipage}{0.8\linewidth}
\begin{center}
\textbf{\texttt{#1}}
\end{center}
\rm
\begin{tabbing}
....\=...\=...\=...\=...\=  \+ \kill
} %
{\end{tabbing}
\end{boxedminipage} \end{center} 
}

\begin{document}

\begin{frontmatter}

\title{Aggregation of Affine Estimators}
\runtitle{Aggregation of Affine Estimators}

 \author{\fnms{Dong} \snm{Dai}\corref{}\ead[label=dong]{dongdai@stat.rutgers.edu},
 \fnms{Philippe} \snm{Rigollet}\thanksref{t2}\ead[label=phil]{rigollet@princeton.edu},
\fnms{Lucy} \snm{Xia}\ead[label=lu]{lxia@princeton.edu}
\and \fnms{Tong} \snm{Zhang}\ead[label=tong]{tzhang@stat.rutgers.edu}
}

 \affiliation{Rutgers and Princeton University}

\thankstext{t2}{Supported in part by NSF grants DMS-1317308  and CAREER-DMS-1053987.}

\address{{Dong Dai}\\
{Statistics department} \\
{Rutgers}\\
{Piscataway,NJ 08854, USA}\\
 \printead{dong}}
\address{{Philippe Rigollet}\\
{Department of Operations Research} \\
{ and Financial Engineering}\\
{Princeton University}\\
{Princeton, NJ 08544, USA}\\
 \printead{phil}}
\address{{Tong Zhang}\\
{Statistics department} \\
{Rutgers}\\
{Piscataway,NJ 08854, USA}\\
 \printead{tong}}
\address{{Lucy Xia}\\
{Department of Operations Research} \\
{ and Financial Engineering}\\
{Princeton University}\\
{Princeton, NJ 08544, USA}\\
 \printead{lu}}
%

\runauthor{Dai et al.}

\begin{abstract}
\ We consider the problem of aggregating a general collection of affine estimators for fixed design regression. Relevant examples include some commonly used statistical estimators such as least squares, ridge and robust least squares estimators. Dalalyan and Salmon \cite{DalSal12} have established that, for this problem, exponentially weighted (EW) model selection aggregation leads to sharp oracle inequalities in expectation, but similar bounds in deviation were not previously known. While results \cite{DaiRigZha12} indicate that the same aggregation scheme may not satisfy sharp oracle inequalities with high probability, we prove  that a weaker notion of oracle inequality for EW that holds with high probability. Moreover,  using a generalization of the newly introduced $Q$-aggregation scheme we also prove sharp oracle inequalities that hold with high probability. Finally, we apply our results to universal aggregation and show that our proposed estimator leads simultaneously to all the best known bounds for aggregation, including $\ell_q$-aggregation, $q \in (0,1)$, with high probability.
\end{abstract}


\end{frontmatter}

\section{Introduction}

\setcounter{equation}{0}
In the Gaussian Mean Model (GMM), we observe a Gaussian random vector $Y \in \R^n$ such that $Y\sim\cN(\mu, \sigma^2I_n)$ where the mean $\mu \in \R^n$ is unknown and the variance parameter $\sigma^2$ is known. For the purpose of discussion, we assume that $\sigma^2=1$ throughout this introduction but our main subsequent results explicitly depend on $\sigma^2$.

This apparently simple model introduced in a notorious paper~\cite{Ste56} by Stein, was the starting point of a vast literature on shrinkage~\cite{Gru98} that later evolved in the Gaussian sequence model. This literature is much too vast to explore here but we refer the reader to the excellent manuscript by Johnstone~\cite{Joh11} for both motivation and partial literature review.

Independently of the variety of methods and results dedicated to the GMM, Nemirovski~\cite{JudNem00,Nem00} introduced \emph{aggregation theory} as a versatile tool for adaptation in nonparametric estimation~\cite{Lec07a,RigTsy07,Yan04}, but also more recently in high dimensional regression \cite{LeuBar06, RigTsy11, DalSal12}. In all these results, exponential weights have played a key role (see \cite{RigTsy12} for a recent survey). Specifically, we focus here on \emph{model selection aggregation} where, given a family of estimators $\hat \mu_1, \dots, \hat \mu_M$, the goal is to mimic the best of them. Originally, aggregation was accompanied  with a sample splitting scheme in which the sample was split into two parts: the first one to construct various estimators and the second to aggregate them. For example, this approach was practically implemented in~\cite{RigTsy07} for density estimation and in \cite{Lec07a} for classification. The advantage of sample splitting is that it allows to \emph{freeze} the first sample and therefore treat the estimators to be aggregated as deterministic functions that only satisfy mild boundedness assumption. This is the framework of \emph{pure aggregation} under which most of the developments have been made starting from the seminal works on aggregation~~\cite{JudNem00,Nem00, Tsy03}. Pure model selection aggregation in the GMM can be described as follows. Given $M \ge 2$ vectors $\mu_1, \ldots, \mu_M$,  the goal is to construct an estimator $\hat \mu$ called \emph{aggregate}, using the observation $Y$ and such that
$$
\|\hat \mu -\mu\|^2 - \min_{1\le j \le M}\|\mu_j -\mu\|^2
$$
is as small as possible, where $\|\cdot\|$ denotes the Euclidean distance on $\R^n$. Bounds on this quantity are called \emph{sharp oracle inequalities}. While not directly connected to Stein's original result on admissibility~\cite{Ste56}, it turns out that for aggregation too, the most natural choice $\hat \mu=\mu_{\hat \jmath}$ where $\hat \jmath=\argmin_{1\le j \le M}\|\mu_j-Y\|^2$ is suboptimal. Nevertheless, this problem is by now well understood and various optimal choices for $\hat \mu$  relying on model averaging rather than model selection were proposed and proved to be optimal (see~\cite{RigTsy12} and references therein). Two approaches have been employed successfully. The first family of methods is based on exponential weights \cite{DalTsy07, DalTsy08}. Following original ideas of Catoni~\cite{Cat99} and Yang~\cite{Yan99}, it can be proved that for any \emph{prior} probability distribution $\pi=(\pi_1, \ldots, \pi_M)$ on $[M]=\{1, \ldots, M\}$, there exists an aggregate $\hat \mu^\mathsf{EW}$ based on exponential weights that satisfies the following \emph{sharp oracle inequality}:
\begin{equation}
\label{EQ:bdEXP}
\E\|\hat \mu^\mathsf{EW} -\mu\|^2 \le \min_{1\le j \le M}\Big\{\|\mu_j -\mu\|^2 + C \log\big(\frac{1}{\pi_j}\big)\Big\}\,,
\end{equation}
where here and in what follows $C>0$ is a numerical constant that may change from line to line. In particular, if $\pi$ is chosen to be the uniform distribution, this estimator attains the optimal rate $C\log(M)$~\cite{RigTsy11} that is independent of the dimension $n$. Nevertheless, it was observed in~\cite{DaiRigZha12} that the random quantity $\|\hat \mu^\mathsf{EW} -\mu\|^2$ may have fluctuation of order $\sqrt{n}$ around its expectation so that the bound~\eqref{EQ:bdEXP} may be fail to accurately describe the risk of $\hat \mu^\mathsf{EW}$, especially for large dimension $n$. To overcome this limitation, a new method called $Q$-aggregation was recently proposed and studied in several settings \cite{Rig12, DaiRigZha12, LecRig13}. It enjoys the following property. For any prior $\pi$ on $[M]$, it yields an aggregate $\hat \mu^Q$ that satisfies not only a sharp oracle inequality \emph{in expectation} of form~\eqref{EQ:bdEXP} but also one that holds \emph{with high probability}:
\begin{equation}
\label{EQ:bdQ}
\|\hat \mu^Q -\mu\|^2 \le \min_{1\le j \le M}\Big\{\|\mu_j -\mu\|^2 + C \log\big(\frac{1}{\delta\pi_j}\big)\Big\}\,,
\end{equation}
with probability $1-\delta$.

In this paper, we extend this work to the aggregation of not fixed vectors $\mu_1, \ldots, \mu_M$ but of affine estimators $\hat \mu_1, \ldots, \hat \mu_M$ that are of the form $\hat \mu_j=A_j Y+b_j$ for some deterministic matrix-vector pair $(A_j, b_j)$. Note that these estimators are constructed using the same observations $Y$ as the ones employed for aggregation. In particular, no sample splitting scheme is needed.

A canonical example of affine estimators where $A_j$ are projection matrices, was first introduced in~\cite{LeuBar06} and further studied in~\cite{RigTsy11} under the light of high-dimensional linear regression. In a remarkable paper, Dalalayan and Salmon~\cite{DalSal12} recently extended these setups to a more general family of affine estimators, under mild conditions on matrices $A_j$. Nevertheless, all these previous papers are limited to deriving sharp oracle inequalities in expectation of the same type as~\eqref{EQ:bdEXP}. Moreover, the lower bounds of \cite{DaiRigZha12} indicate that the estimators based on exponential weights that are employed in~\cite{LeuBar06, RigTsy11, DalSal12} are unlikely to satisfy sharp oracle inequalities with high probability. In this paper, akin to \cite{Rig12, DaiRigZha12, LecRig13}, we demonstrate that $Q$-aggregation succeeds where exponential weights have failed by proving a sharp oracle inequality that holds with high probability in section~\ref{SEC:sharp}. Yet, the situation regarding exponential weights is not desperate as we show in section~\ref{SEC:weak} that it still leads to a weaker notion of oracle inequalities.

The rest of this paper is organized as follows. In the next section, we give a precise description of the problem of \emph{model selection aggregation of affine estimators} and give a solution to this problem using $Q$-aggregation. Specifically, in section~\ref{SEC:sharp}, we show that   for any \emph{prior} probability distribution $\pi=(\pi_1, \ldots, \pi_M)$ on $[M]=\{1, \ldots, M\}$, there exists an aggregate $\hat \mu^{Q}$ based on $Q$-aggregation that satisfies a sharp oracle inequality of the form
\begin{equation}
\label{EQ:bdQ_P}
\|\hat \mu^Q -\mu\|^2 \le \min_{j \in [M]}\Big\{\|\hat \mu_j -\mu\|^2 + C \log\big(\frac{1}{\pi_j}\big)+ 4\sigma^2\tr(A_j)\Big\}\,,
\end{equation}
that holds both in expectation and with high probability, where $\tr(A_j)$ denotes the trace of $A_j$.
We continue by proving  in section~\ref{SEC:weak} that for any $\eps>0$, there exists a choice of the temperature parameter for which the better known aggregate $\hat \mu^\mathsf{EW}$ based on exponential weights  satisfies a \emph{weak oracle inequality} that holds with high probability
\begin{equation}
\label{EQ:bdEXP_P}
\|\hat \mu^\mathsf{EW} -\mu\|^2 \le \min_{j \in [M]}\Big\{(1+\eps)\|\hat \mu_j -\mu\|^2 + \frac{C}{\eps} \log\big(\frac{1}{\pi_j}\big)+ 8\sigma^2\tr(A_j)\Big\}\,.
\end{equation}
 Such an inequality completes the sharp oracle inequality of \cite{DalSal12} that holds in expectation.

We give applications of these oracle inequalities to sparsity pattern aggregation and universal aggregation in section~\ref{SEC:applis}. In particular, we show that $Q$-aggregation of projection estimators leads to the first sharp oracle inequalities \emph{that hold with high probability} for these two problems. By ``high probability", we mean a statement that holds with probability at least $1-\delta$, $0<\delta<1/2$. Our results below exhibit explicit dependence on $\delta$.

\medskip

\noindent{\sc Notation:} For any integer $n$, the set of integers $\{1, \ldots, n\}$ is denoted by $[n]$. We denote by $\tr(A)$ and $\rk(A)$ respectively the  trace and the rank of a square matrix $A$. We denote by $\|\,\cdot\,\|$ the Euclidean norm of $\R^n$ and by $|J|$ the cardinality of a finite set $J$. For any real numbers $a_1, \ldots, a_n$, $\diag(a_1, \ldots, a_n)$ denotes the $n \times n$ diagonal matrix with  $a_1, \ldots, a_n$, on the diagonal. The indicator function is denoted by $\1(\cdot)$ and for any integer $n$, $K \subset [n]$, $\1_K$ denotes the vector  $v\in \{0,1\}^n$ with $j$th coordinate given by $v_j=1$ iff $j \in K$.
For any matrix $B$, $B^\dagger$ denotes the Moore-Penore pseudoinverse of $B$. The operator norm of a matrix is denoted by $\|\cdot\|_{\mathsf{op}}$. The cone of $n\times n$ positive semidefinite matrices is denoted by $\cS_n$. The flat simplex of $\R^M$ is denoted by $\Lambda^M$ and is defined by
$$
\Lambda_M=\Big\{ \theta \in \R^M\,:\, \theta_j \ge 0\,, \sum_{j=1}^M \theta_j=1\Big\}
$$
The set $\Lambda_M$ can be identified to the set of probability measures on $[M]$ and for any $\theta, \pi \in \Lambda_M$, we define the Kullback-Leibler divergence between these two measures by
$$
\cK(\theta, \pi)=\sum_{j=1}^M \theta_j \log\big(\frac{\theta_j}{\pi_j}\big)\,,
$$
with the usual convention that $0\log(0)=0, 0\log(0/0)=0$ and $\theta \log(\theta/0)=+\infty, \ \forall\, \theta>0$. Finally, throughout the paper, we use the notation $\blog(x)$ to denote the function $\blog(x)=(\log x)\vee 1$.

\section{Aggregation of affine estimators}
\setcounter{equation}{0}
Recall that the Gaussian Mean Model (GMM) can be written as follows. One observes $Y \in \R^n$ such that
\begin{equation}
\label{EQ:GMM}
Y =\mu +\xi\,,\qquad \xi \sim \cN(0,I_n)\,.
\end{equation}
Throughout this paper and in accordance with \cite{DalSal12}, we call an affine estimator of $\mu$ any estimator $\hat \mu$ of the form
\begin{equation}
\label{EQ:def_affine}
\hat \mu=AY + b\,,
\end{equation}
where $A \in \cS_n$ is a $n \times n$ matrix and $b \in \R^n$ is a $n$-dimensional vector. Both $A$ and $b$ are deterministic.

Given a family of affine estimators $\hat \mu_1, \ldots, \hat \mu_M$, where $\hat \mu_j=A_jY+b_j$ and a prior probability measure $\pi=(\pi_1, \ldots, \pi_M)$ on these estimators, our goal is to construct an aggregate $\tilde \mu$ such that
\begin{equation}
\label{EQ:OIaffine}
\|\tilde\mu  -\mu\|^2 \le (1+\eps)\|\hat \mu_j -\mu\|^2 + C\big[ \log\big(\frac{1}{\delta\pi_j} \big)+ T_j\big]\,,
\end{equation}
with probability $1-\delta$ for any $j\in [M]$, where $T_j>0$ is as small as possible and $\eps\ge 0$. As we will see, we can achieve $\eps=0$ using $Q$-aggregation  but only prove a weak oracle inequality with $\eps>0$ in section~\ref{SEC:weak} using exponential weights.

Inequalities of the form~\eqref{EQ:OIaffine} with $\eps>0$ can be of interest as long as there exists a candidate affine estimator $\hat\mu_j$ that is close to $\mu$ with high probability. Several examples where it is the case are described in \cite{DalSal12}.

Our results below hold under the following general condition on the family of matrices $\{A_j\}_{j \in [M]}$.
\begin{condition}
\label{cond1}
There exists a finite $V>0$ such that  $\max_{j \in [M]}\|A_j\|_{\mathsf{op}}=V$.
\end{condition}

To illustrate the purpose of aggregating affine estimators and the relevance of Condition~\ref{cond1}, observe that a large body of the literature on the GMM studies estimators of the form $AY$, where $A=\diag(a_1, \ldots, a_n)$ is a diagonal matrix with elements $a_j \in [0,1]$ for all $j=1, \ldots, n$. If $\mu$ is assumed to belong to some family of regularity classes such as Sobolev ellipsoids, Besov classes, tail classes, it has been proved that such estimators are  minimax optimal (see \cite{CavTsy01, Tsy09, Joh11}). Commonly used examples include  ordered projection estimators, spline estimators and Pinsker estimators (see \cite{DalSal12} for a detailed description).
These estimators are known to  be minimax optimal over Sobolev ellipsoids  \cite{Pin80, GolNus92, Tsy09}. Diagonal filters trivially satisfy Condition~\ref{cond1} with $V=1$.

We give details of a specific application to sparsity pattern aggregation and its consequences on universal aggregation in section~\ref{SEC:applis}.

\subsection{Sharp oracle inequalities using $Q$-Aggregation}
\label{SEC:sharp}
In this section, we state our main result: a sharp oracle inequality for an aggregate of affine estimators based on $Q$-aggregation.
Specifically, we consider the problem of aggregating general affine estimators $\hat \mu_j=A_j Y+b_j, j \in [M]$ that satisfy Condition~\ref{cond1}. Note that unlike \cite{DalSal12}, we do not require that matrices $A_j, j \in [M]$  commute and we make no assumption on the vectors $b_j, j=1,\ldots, M$. Moreover, our results can be extended to an infinite family $\{(A_\lambda, b_\lambda), \lambda \in \Lambda\}$ as in \cite{DalSal12} but we prefer to present our result in the discrete case for the sake of clarity.

For any $\theta \in \R^M$, let $\mu_{\theta}$ denote the linear combination of some given affine estimators $\hat \mu_1, \ldots, \hat \mu_M$ that is defined by
$$
\mu_{\theta}=\sum_{j=1}^M \theta_j\hat  \mu_j\,.
$$
Our goal is to find a vector $\hat \theta \in \R^M$ such that the aggregate $\mu_{\hat \theta}$  mimics the affine estimator $\hat \mu_j$ that is the closest to the true mean $\mu$.

In this paper, we consider a generalization of the $Q$-aggregation scheme of static models that was developed in \cite{Rig12, DaiRigZha12}. To that end, fix a prior probability distribution $\pi \in \Lambda_M$ and for any $\theta \in \Lambda_M$, define
\begin{equation}
\label{EQ:defQ}
Q(\theta) = \nu\sum_{j=1}^M\theta_k \|Y-\hat \mu_j\|^{2}+(1-\nu)\|Y-\mu_{\theta}\|^{2}+ \sum_{j=1}^M\theta_j C_j +\lambda\cK(\theta,\pi),
\end{equation}
where $\nu \in (0,1)$ and $\lambda>0$ are tuning parameters, and $C_j$ is set to be
\begin{equation}
\label{EQ:defC_j}
C_j=4\sigma^2\tr(A_j)\,.
\end{equation}
Let now $\hat \theta$ be defined as
\begin{equation}
\label{EQ:defQS}
\hat{\theta}\in\argmin_{\theta\in\Lambda_M} Q(\theta)\,.
\end{equation}
The resulting estimator $\mu_{\hat \theta}$ is called \emph{$Q$-aggregate} estimator of $\mu$.
Theorem \ref{thm:T2} is our main result.
\begin{theorem}\label{thm:T2}
Consider the GMM~\eqref{EQ:GMM} and let $\hat \mu_j=A_jY+b_j, j \in [M]$ be affine estimators of $\mu$ together with a prior distribution $\pi=(\pi_1, \ldots, \pi_M)$ on these estimators and let $V=\max_{j \in [M]}\|A_j\|_{\mathsf{op}}$.
Let $\hat \mu^Q=\mu_{\hat \theta}$ be the $Q$-aggregate estimator with $\hat \theta$ defined in \eqref{EQ:defQS} with tuning parameters $\nu \in (0,1)$ and
$\lambda \ge \frac{8\sigma^2}{\min(\nu, 1-\nu,(\frac{5}{2}V)^{-1})}$. Then for any $\delta>0$, any $j\in [M]$, with probability at least $1-\delta$, we have
\begin{equation}
\label{EQ:OIQP}
\|\hat \mu^Q-\mu\|^2
\leq \min_{j \in [M]}\Big\{\|\hat \mu_j-\mu\|^2+C_j+\lambda\log(\frac{1}{\pi_j\delta})\Big\}\,.
\end{equation}
Moreover, the same $Q$-aggregate estimator $\hat \mu^Q$ satisfies
\begin{equation}
\label{EQ:OIQE}
\E\|\hat \mu^Q-\mu\|^2
\leq \min_{j \in [M]}\Big\{\E\|\hat \mu_j-\mu\|^2+C_j+\lambda\log(\frac{1}{\pi_j})\Big\}.
\end{equation}
\end{theorem}
A few remarks are in order. Note that the oracle inequality of Theorem~\ref{thm:T2} is \emph{sharp} since the leading term $\|\hat \mu_j-\mu\|^2$ has multiplicative constant $1$. A similar oracle inequality was obtained in \cite{DalSal12} bit our main theorem above presents significant differences. First, and this is the main contribution of this paper, the above oracle inequality holds with high probability whereas the ones in \cite{DalSal12} only hold in expectation. Nevertheless, our model is simpler than the one studied in \cite{DalSal12} who study heteroskedastic regression. Moreover, the bound in \cite[Theorem~2]{DalSal12} is ``scale-free" whereas ours depends critically on the size of the matrices $A_j$ via $C_j$ and $V$. We believe that this dependence cannot be avoided in high probability bounds as such quantities essentially control the deviations of estimators. As we will see, the bounds of Theorem~\ref{thm:T2} are sufficient to perform  sparsity pattern aggregation and universal aggregation optimally.

%
%

\subsection{Weak oracle inequality using exponential weights}
\label{SEC:weak}

The oracle inequalities~\eqref{EQ:bdEXP} and~\eqref{EQ:bdQ} are \emph{sharp} in contrast to \emph{weak} oracle inequalities where the right-hand side of~\eqref{EQ:bdEXP} or~\eqref{EQ:bdQ} is replaced by
$$
\min_{1\le j \le M}\{(1+\eps)\|\mu_j -\mu\|^2 + C_j+\frac{C}{\eps} \log\big(\frac{1}{\delta\pi_j}\big)\}\,,
$$
for some $\eps>0$ (see \cite{LecMen12a} for a discussion on the difference between sharp and weak oracle inequalities). While they appear to be quite similar, some estimators do satisfy weak oracle inequalities while they do not satisfy sharp ones. This is the case of the aggregate with exponential weight that provably fails to satisfy a sharp oracle inequality with high probability in a certain setups~\cite[Proposition~2.1]{DaiRigZha12}.


To prove weak oracle inequalities that hold with high probability, we  modify the aggregate studied in \cite{DalSal12}. Recall that $\hat \mu_j=A_jY+b_j, j\in [M]$ is a family of affine estimators equipped with a prior probability distribution $\pi \in \Lambda_M$ and that $C_j$ is defined in~\eqref{EQ:defC_j}. Let $\hat \theta \in \Lambda_M$ be the vector of exponential weights defined by
\begin{equation}
\label{EQ:defEXP}
\hat{\theta}_j \propto  \pi_j \exp\big(-\frac{\|Y - \hat \mu_j\|^2 +C_j}{\lambda}\big)\,, \quad j \in [M]\,.
\end{equation}
The parameter $\lambda>0$ is often referred to as \emph{temperature parameter}. It is not hard to show (see, e.g., \cite[p. 160]{Cat04}) that $\hat{\theta}$ is the solution of the following optimization problem:
\begin{equation}
\label{EQ:defEXP2}
\hat{\theta}\in\argmin_{\theta\in\Lambda}\Big\{\sum_{j=1}^M\theta_j\|Y-\hat \mu_j\|^2+\sum_{j=1}^M\theta_j C_j+\lambda\mathcal{K}(\theta,\pi)\Big\}\,.
\end{equation}
Observe that the above criterion corresponds to $Q$ defined in~\eqref{EQ:defQ} with $\nu=1$, that is without the quadratic term in $\theta$. We believe that this quadratic term is key in obtaining sharp oracle inequalities that hold with high probability. We already know from previous work \cite{LeuBar06, RigTsy11, RigTsy12, DalSal12} that this term is not necessary to obtain sharp oracle inequalities that hold in expectation. As illustrated  below, it is also not required to get weak oracle inequalities, even with high probability.

Denote by $\hat \mu^{\mathsf{EW}}=\sum_{j=1}^M \hat \theta_j \hat \mu_j$ the aggregate with exponential weights $\hat \theta_j\,, j \in [M]$ defined in~\eqref{EQ:defEXP}.
\begin{theorem}\label{thm:T1}
Let the conditions of Theorem~\ref{thm:T2} hold. Let $\hat \mu^{\mathsf{EW}}=\mu_{\hat \theta}$ be the aggregate  with exponential weights $\hat \theta$ defined in \eqref{EQ:defEXP} with tuning parameter
$\lambda \geq  4\sigma^2(16\vee 5V)$. Then for any $\delta>0$, with probability at least $1-\delta$, we have
\[
\|\mu-\hat \mu^{\mathsf{EW}}\|^2 \leq \min_{j \in [M]}\Big\{\big(1+\frac{128\sigma^2}{3\lambda}\big)\|\mu-\hat\mu_j\|^2+3\lambda\log\big(\frac{1}{\delta\pi_j}\big)+2 C_j\Big\}.
\]
\end{theorem}
Note that unlike Theorem~\ref{thm:T2}, the right-hand side of the above oracle inequality is multiplied by a factor $1+\varepsilon>1$: it is a weak oracle inequality but it holds with high probability and thus complements the results of~\cite{DalSal12} on aggregation of affine estimators using exponential weights. Alquier and Lounci \cite{AlqLou11} prove  the first oracle inequality with high probably using exponential weights. They use specific projection estimators  $\hat \mu_j$ for sparsity pattern aggregation but make extra assumptions and use a prior probability measure tailored to these assumptions in order to obtain a sharp oracle inequality. While their final result \cite[Theorem~3.1]{AlqLou11} is not directly comparable to ours, a weak oracle inequality similar to the one above can be deduced from their proof. Actually, our proof uses one of their arguments.

\section{Sparsity pattern aggregation}
\label{SEC:applis}
\setcounter{equation}{0}
In this section, we  illustrate the power of the two oracle inequalities stated in the previous section. Indeed, carefully selecting the affine estimators $\hat \mu_1, \ldots, \hat \mu_M$, as well as the the prior probability distribution $\pi$ leads to various optimal results. Some results for diagonal filters can be found in \cite{DalSal12} and we focus here on sparsity pattern aggregation.

Recall the results we have proved in the previous section. With probability at least $1-\delta$, for $\lambda$ large enough and,
$$
\|\mu_{\hat \theta}-\mu\|^2
\leq  \min_{j \in [M]}\Big\{\big(1+\frac{t}{\lambda}\big)\|\hat \mu_j-\mu\|^2+C_j+\lambda\log(\frac{1}{\pi_j\delta})\Big\}\,.
$$
where $t=0$ if $\hat \theta$ is computed according to~\eqref{EQ:defQS} and $t=4\sigma^2$ if $\hat \theta$ is computed according to~\eqref{EQ:defEXP}.

In the sequel, we fix $\nu=1/2$ in $Q$-aggregation since this choice leads to the sharpest bounds.
\subsection{Sparsity pattern aggregation}
\label{SEC:sparse}

Let $X_1, \ldots, X_p \in \R^n$ be given vectors and assume that $\mu \in \R^n$ can be well approximated by a  linear combination of $X_j, j \in J^*$ for some unknown sparsity pattern $J^*\subset [p]$. More precisely, we are interested in \emph{sparse} linear regression, where the goal is to find a sparse $\beta \in \R^p$ such that $\|\X\beta-\mu\|^2$ is small, where $\X =[X_1, \ldots, X_p]$ is the $n \times p$ design matrix obtained by concatenating the $X_j$'s. Akin to \cite{BicRitTsy09, RigTsy11}, we do not assume that there exists a sparse $\beta^*$ such that $\X\beta^*=\mu$ but rather that there may be a systematic error. Oracle inequalities such as the ones described below in Theorems~\ref{TH:spa} and~\ref{TH:spaellq} capture the statistical precision of fitting possibly misspecified sparse linear models in the GMM.

To achieve our goal, we follow the same idea as in \cite{RigTsy11, RigTsy12} and employ \emph{sparsity pattern aggregation}. The idea can be summarized as follows. For each sparsity pattern of $\beta$, compute the least squares estimator and then aggregate these (projection) estimators. Specifically, for each sparsity pattern $J \subset [p]$  define $\X_J$ to be the $n \times |J|$ matrix obtained by concatenating $X_j, j \in J$ and let $A_J=\X_J(\X_J^\top \X_J)^\dagger \X^\top_J$ denote the projection matrix onto the linear span $\linspan(\X_J)$ of $X_j, j \in J$.  If $J^*$ was known, a good candidate to estimate $\mu$ would be the least squares estimator $\hat \mu_{J^*}=A_{J^*} Y$. Since $J^*$ is unknown, we propose to aggregate the affine (actually linear) estimators $\hat \mu_{J}=A_J Y, J\subset [p]$. This approach is called \emph{sparsity pattern aggregation} \cite{RigTsy11} and can be extended to more general notions of sparsity such as group sparsity or fused sparsity  \cite{RigTsy12}. It yields a family of affine estimators $\hat \mu_J=A_JY$ such that $C_J=4\sigma^2$,$\tr(A_J)=4\sigma^2\rk(\X_J)$ and $V=\max_J\|A_J\|_\mathsf{op}=1$.

Sparsity pattern aggregation has been shown to attain the best available sharp oracle inequalities in expectation \cite{RigTsy11, RigTsy12} and one of the main contribution of this paper is to extend these results to results with high probability. Moreover, it leads to universal aggregation with high probability (see section~\ref{sec:univ}).

The key to sparsity pattern aggregation is to employ a correct prior probability distribution. Rigollet and Tsybakov~\cite{RigTsy12}, following \cite{LeuBar06, Gir08} suggest to use
\begin{equation}
\label{EQ:priorSPA}
\pi_J\propto \frac{e^{-|J|}}{{p \choose |J|}}\,.
\end{equation}
In particular, it exponentially downweights patterns $J$ according to their cardinality.

For any $\beta \in \R^p\setminus\{0\}$, let $|\beta|_0$ denote the number of nonzero coefficients of $\beta$ and, by convention, let $|0|_0=1$.
\begin{theorem}
\label{TH:spa}
Let $\hat \mu_J, J \subset [p]$ be the least squares estimator defined as above, let $\pi$ be the sparsity prior defined in~\eqref{EQ:priorSPA} and fix $\delta>0$. Then the following statements hold:
\begin{itemize}
\item[(i)] For $\lambda\geq 20\sigma^2$, with probability at least $1-\delta$, the $Q$-aggregate estimator $\hat \mu^Q$ satisfies
\begin{equation}
\label{EQ:soiQ}
\|\hat \mu^Q-\mu\|^2 \le \min_{\beta \in \R^p}\Big\{\|\X\beta-\mu\|^2+  3(\lambda+2\sigma^2)|\beta|_0\log\big(\frac{2ep}{|\beta|_0\delta}\big)\Big\}.
\end{equation}
\item[(ii)] For $\lambda\geq 64\sigma^2$, with probability at least $1-\delta$, the aggregate with exponential weights $\hat \mu^{\mathsf{EW}}$ satisfies
\begin{equation}
\label{EQ:soiEXP}
\|\hat \mu^\mathsf{EW}-\mu\|^2 \le \min_{\beta \in \R^p}\Big\{\big(1+\frac{128\sigma^2}{3\lambda}\big)\|\X\beta-\mu\|^2+   6(\lambda+2\sigma^2)|\beta|_0\log\big(\frac{2ep}{|\beta|_0\delta}\big)\Big\}.
\end{equation}
\end{itemize}
\end{theorem}

\begin{corollary} Taking $\lambda=20\sigma^2$ and $\lambda= 64\sigma^2$ for  $\hat \mu^Q$ and $\hat \mu^\mathsf{EW}$ respectively, with probability at least $1-\delta$, we have:
\begin{equation}
\label{EQ:soiQC}
\|\hat \mu^Q-\mu\|^2 \le \min_{\beta \in \R^p}\big\{\|\X\beta-\mu\|^2+  66\sigma^2|\beta|_0\log\big(\frac{2ep}{|\beta|_0\delta}\big)\big\},
\end{equation}
and
\begin{equation}
\label{EQ:soiEXPC}
\|\hat \mu^\mathsf{EW}-\mu\|^2 \le \min_{\beta \in \R^p}\big\{\frac53\|\X\beta-\mu\|^2+  396\sigma^2|\beta|_0\log\big(\frac{2ep}{|\beta|_0\delta}\big)\big\}.
\end{equation}
\end{corollary}

The novelty of this result is twofold. First, we use $Q$-aggregation to obtain the first sharp sparsity oracle inequalities that hold with high probability under no additional condition on the problem. Second, we prove a weak sparsity oracle inequality for the aggregate based on exponential weights that holds with high probability. While it is only a weak oracle inequality, it extends the results of Rigollet and Tsybakov \cite{RigTsy11, RigTsy12} that hold only in expectation and the results of \cite{AlqLou11} that hold with high probability but under additional conditions.
%
%
%
%
%
%
%
\subsection{$\ell_q$-aggregation}

Recently, Rigollet and Tsybakov \cite{RigTsy11} observed that any estimator that satisfies an oracle inequality such as~\eqref{EQ:OIQE} also adapts to sparsity when measured in terms of $\ell_1$ norm. Specifically, their result \cite[Lemma~A.2]{RigTsy11} implies that if $\max_j \|\mu_j\|\le \sqrt{n}$, then for any constant $\nu>0$, it holds
\begin{equation}
\label{EQ:lemA2}
\min_{\theta \in \R^M} \Big\{ \|\mu_\theta -\mu\|^2 + \nu^2|\theta|_0\log\big(1+\frac{eM}{|\theta|_0}\big)\Big\}\le \min_{\theta \in \sB_1(1)}  \|\mu_\theta -\mu\|^2 +\bar c \nu \sqrt{n \log\big(1+\frac{eM\nu}{\sqrt{n}}\big)}\,,
\end{equation}
where $\bar c$ is an absolute constant. The above bound hinges on a Maurey argument, which, as noticed by Wang {\it et al.} \cite{WanPatGao11}, can be  extended from  $\ell_1$ balls to $\ell_q$ balls for $q \in (0,1]$. It has been argued that $\ell_q$-balls ($0<q\leq1$) describe vectors that are ``almost sparse'' \cite{FouPajRau10, Joh11}.

For any $q \in (0,1], \theta \in \R^M$, let $|\theta|_q$ denote the $\ell_q$-``norm" of $\R^M$ of $\theta$ defined  by
$$
|\theta|_q=\Big( \sum_{j \in [M]} |\theta_j|^q\Big)^\frac1q\,.
$$
Moreover, for a given radius $R>0$ and any $q \in [0,1]$, define the  $\ell_q$-ball of radius $R$ by
$$
\sB_q(R)=\big\{\theta \in \R^M\,:\, |\theta|_q\le R\big\}\,.
$$

Not surprisingly, these almost sparse vectors can be well approximated by sparse vectors as illustrated in the following lemma that generalized~\eqref{EQ:lemA2}

\begin{lemma}
\label{LEM:maurey2}
Fix $\nu>0$, $M\ge 3$ and let and $\mu_j\,, j \in [M]$ such that $\max_j\|\mu_j\|^2\le B^2$. Then
$$
\min_{\theta \in \R^M} \Big\{ \|\mu_\theta -\mu\|^2 + \nu^2|\theta|_0\log\big(1+\frac{eM}{|\theta|_0}\big)\Big\}\le
\inf_{0\le q\le 1}\min_{\theta \in \R^M}  \Big\{ \|\mu_\theta -\mu\|^2 + \varphi_{q,M}(\theta; \nu, B)\Big\}\,,
$$
where
\begin{equation}
\label{EQ:defPHI}
\varphi_{q,M}(\theta; \nu, B)=9\nu^{2-q}|\theta|_q^qB^q \Big[\blog\Big(\frac{eM\nu^q}{B^q |\bar\theta|^q_q\delta}\Big)\Big]^{1-\frac{q}{2}}
 \vee 3\nu^2\blog\big(\frac{eM}{\delta}\big).
\end{equation}
with the convention $|\theta|_0^0=|\theta|_0$.
\end{lemma}
We postpone the proof to Appendix~\ref{SEC:maurey} where further results on the approximation of vectors with small $\ell_q$ norm by sparse vectors, can be found. We are now in a position to state the main result of this subsection. Its proof follows directly from the above lemma by rounding $\sqrt{66}$ up to $9$ and $\sqrt{5\cdot 396/3}$ to $16$.
\begin{theorem}
\label{TH:spaellq}
Let $\hat \mu_J, J \subset [p]$ be defined as in subsection~\ref{SEC:sparse} with $\pi$ being the sparsity prior defined in~\eqref{EQ:priorSPA}. Moreover, assume that $\max_j\|X_j\|^2\le B^2$ for some $B>0$ and assume that $M \ge 3$. Then,  the following statements hold with probability at least $1-\delta$:
\begin{itemize}
\item[(i)] The $Q$-aggregate estimator $\hat \mu^Q$ with $\lambda=20\sigma^2$ satisfies
\begin{equation}
\label{EQ:soiQ2}
\|\hat \mu^Q-\mu\|^2 \le \inf_{0\le q\le 1}\min_{\beta \in \R^p}\big\{\|\X\beta-\mu\|^2+  \varphi_{q,p}(\theta; 9\sigma, B)\big\}\,.
\end{equation}
\item[(ii)] The aggregate with exponential weights $\hat \mu^{\mathsf{EW}}$ with $\lambda=64\sigma^2$  satisfies
\begin{equation}
\label{EQ:soiEXP2}
\|\hat \mu^\mathsf{EW}-\mu\|^2 \le \frac53\inf_{0\le q\le 1}\min_{\beta \in \R^p}\big\{\|\X\beta-\mu\|^2+   \varphi_{q,p}(\theta; 16\sigma, B)\big\}\,,
\end{equation}
where, in both cases, $\varphi_{q,p}$ is defined in~\eqref{EQ:defPHI}\,.
\end{itemize}
\end{theorem}
Both~\eqref{EQ:soiQ2} and~\eqref{EQ:soiEXP2} can be compared to the prediction rates over $\ell_q$ balls that were derived in \cite{RasWaiYu12} where the setup is the following. First, it is assumed that  the true mean $\mu$ in~\eqref{EQ:GMM} is of the form $\mu=\X\beta^*$ for some $\beta^* \in \sB_q(R), R>0$ and that $B=\kappa\sqrt{n}$. In this case, it follows from Theorem~\ref{TH:spaellq} that  with probability at least $1-\delta$, we have for any $\tilde \mu \in \{\hat \mu^Q, \hat \mu^{\mathsf{EW}}\}$ that
$$
\max_{\beta^* \in \sB_q(R)}\frac1n\|\tilde \mu -\X\beta^*\|^2\le C_1\kappa^2R^q\Big[\frac{\sigma^2}{\kappa^2}\frac{\log\Big(\frac{ep}{\delta}\big(\frac{\sigma}{R\kappa\sqrt{n}}\big)^q \Big) }{n} \Big]^{1-\frac{q}2}\vee C_2\sigma^2\log\big(\frac{ ep}{\delta}\big)\,,
$$
for some numerical constants $C_1, C_2$. In their specific regime of parameters, our rates are of the same order as \cite[Theorem~4]{RasWaiYu12} and are therefore optimal in that range. However, we provide a better finite sample performance and explicit dependence in the confidence parameter $\delta$ as well as explicit constants that do not depend on $q$. In particular, our bounds are continuous functions of $q$ on the whole closed interval $[0,1]$.    More strikingly, unlike \cite{RasWaiYu12} neither of the estimators $\hat \mu^Q, \hat \mu^{\mathsf{EW}}$ depends on $q$ or $R$ and yet they optimally adapt to these parameters. This remarkable phenomenon is even better illustrated in the context of \emph{universal aggregation}.

\subsection{Universal aggregation}
\label{sec:univ}
In his original description of aggregation, Nemirovski \cite{Nem00} introduced three types of aggregation to which three new types were added later \cite{BunTsyWeg07b, Lou07, WanPatGao11}. All of these aggregation problems can be described in the following unified way. Given $M\ge 2$ deterministic vectors $\mu_1, \ldots, \mu_M \in \R^n$ and a set $\Theta \subset \R^M$, the goal is to construct an aggregate $\tilde \mu$ such that
\begin{equation}
\label{EQ:oiTheta}
\|\tilde \mu -\mu\|^2 \le \inf_{\theta \in \Theta} \|\mu_\theta -\mu\|^2 + C \Delta_{n,M}(\Theta), \quad C>0
\end{equation}
with high probability and where the remainder term $\Delta_{n,M}(\Theta)>0$ is as small as possible. To each of the six types of aggregation,  corresponds a unique $\Theta \subset \R^M$ and a smallest possible $\Delta_{n,M}(\Theta)$ for which~\eqref{EQ:oiTheta} holds. Such a $\Delta_{n,M}(\Theta)$ is called the
\emph{optimal rate of aggregation} (over $\Theta$) \cite{Tsy03}. The six types of aggregation all correspond to choices of $\Theta$ that are intersections of balls $\sB_q(R)$ for various choices of $q$ and $R$. They are summarized in Table~\ref{TAB:agg}. We add a new natural type of aggregation that we call \emph{$D$-$\ell_q$ aggregation}, where, by analogy to $D$-linear and $D$-convex aggregation, we add to $\ell_q$ aggregation the constraint that $\theta$ must be $D$-sparse. In particular, $D$-convex aggregation introduced in \cite{Lou07} can be identified to $D$-$\ell_1$ aggregation.

\renewcommand{\arraystretch}{1.5}
\begin{table}[h]
\begin{center}
\begin{tabular}{|r|l|c|}
\hline
Type of aggregation & $\Theta$ & Optimal rate \\
\hline
\hline
Model Selection   \cite{Nem00}& $\sB_0(1)\cap\sB_1(1)$ &$\sigma^2\log(\frac{M}{\delta})$ 
\\
Convex  \cite{Nem00}& $\sB_1(1)$ &$\Big[\sigma B\sqrt{\blog(\frac{\sigma M}{\delta B})}\vee\sigma^2\blog(\frac{M}{\delta})\Big]\wedge \sigma^2M \log(\frac1{\delta})$\\
Linear  \cite{Nem00}& $\sB_0(M)=\R^M$ &$\sigma^2M\log(\frac{1}{\delta})$\\
$D$-linear  \cite{BunTsyWeg07b} & $\sB_0(D)$ & $\sigma^2D\log(\frac{M}{\delta D})$\\
$D$-convex  \cite{Lou07} & $\sB_0(D)\cap \sB_1(1)$ &$\big[\sigma B\sqrt{\blog(\frac{\sigma M}{\delta B})}\vee\sigma^2\blog(\frac{M}{\delta})\big]\wedge \sigma^2D\log(\frac{M}{\delta D})$\\
$\ell_q$ \cite{WanPatGao11} & $\sB_q(R)$ &$\Big[\sigma^{2-q}R^qB^q \big[\blog\big(\frac{M}{\delta}\big( \frac{\sigma}{BR}\big)^q\big)\big]^{1-\frac{q}{2}}\vee  \sigma^2 \blog(\frac{M}{\delta})\Big]\wedge \sigma^2 M\log(\frac{1}{\delta})$ \\
$D$-$\ell_q$  & $\sB_0(D)\cap\sB_q(R)$ &$\Big[\sigma^{2-q}R^qB^q \big[\blog\big(\frac{M}{\delta}\big( \frac{\sigma}{BR}\big)^q\big)\big]^{1-\frac{q}{2}}\vee  \sigma^2 \blog(\frac{M}{\delta})\Big]\wedge \sigma^2 D\log(\frac{M}{\delta D})$ \\

\hline
 \end{tabular}
\end{center}
\caption{The seven types of aggregation and the corresponding choice of $\Theta$. The range of parameters is $q \in (0,1), D\in [M], R>0$. All numerical constants have been removed for clarity.} \label{TAB:agg}
\end{table}

While most papers on the subject use different estimators for different aggregation problems \cite{Nem00, Tsy03, RigTsy07, Rig12}, Bunea {\it et al.} \cite{BunTsyWeg07b} were the first to suggest that one single estimator could solve several aggregation problems all at once and used the {\sc bic} estimator to obtain partial results in the form of weak oracle inequalities. More recently, Rigollet and Tsybakov \cite{RigTsy11} showed that the exponential screening estimator solved the first five types of aggregation all at once, without the knowledge of $\Theta$. 
Using similar arguments, we now show that the $Q$-aggregate solves at once, all seven problems of aggregation described in Table~\ref{TAB:agg}, not only in expectation, but also with high probability.

\begin{theorem}
\label{TH:univagg}
Fix, $M\ge 3, n \ge 1, D \in [M], B\ge 1, q \in (0,1), R>0$ and $\delta \in (0,1)$.
 Moreover, fix $\mu_1, \ldots, \mu_M \in \R^n$ such that $\max_j\|\mu_j\|^2 \le B^2$. Then, for $\lambda=20\sigma^2$, the $Q$-aggregate estimator $\hat \mu^Q$ satisfies the following oracle inequalities simultaneously  with probability at least $1-\delta$. For any  $\Theta \in \{\sB_0(1)\cap\sB_1(1),\sB_1(1), \R^M, \sB_0(D), \sB_0(D)\cap \sB_1(1), \sB_q(R), \sB_0(D)\cap\sB_q(R)\}$, it holds
\begin{equation}
\label{EQ:soiQ3}
\|\hat \mu^Q-\mu\|^2 \le \min_{\theta \in \Theta}\|\mu_\theta-\mu\|^2+  C\Delta_{n,M}(\Theta)\,,\quad C>0\,.
\end{equation}
where $\Delta_{n,M}(\Theta)$ is defined in Table~\ref{TAB:agg}.
\end{theorem}

Note that the rates in Table~\ref{TAB:agg} are optimal in the sense of \cite{Tsy03} for the most interesting ranges of parameters. Indeed, they match the most general lower bounds of~\cite{RigTsy11, RasWaiYu12, WanPatGao11} apart from minor discrepancies that can be erased by placing appropriate assumptions on the range of parameters considered. It is not hard to see from our proofs where the ambiant dimension $M$ can be replaced by the dimension of the linear span of $\mu_1, \ldots \mu_M$ should appear in these bounds \cite{RigTsy11}. Since this is not the main focus of our paper, we choose not to have this dependence explicit in our bounds but in view of the similarity of our proof techniques and that of \cite{RigTsy11}, it is clear that it can be made explicit whenever appropriate by a simple modification of the prior $\pi$.

\appendix

\section{Proofs of the main theorems}
\setcounter{equation}{0}
%
%

The following lemma is key to both of our theorems. It allows us to control the deviation of the empirical risk of any aggregate $\mu_{\hat \theta}$ around its true risk.
\begin{lemmapp}\label{LEM:main}
Fix  $\lambda\geq 20V \sigma^2$. Let $\mu_{\hat{\theta}}=\sum_{k \in [M]}\hat\theta_k \hat \mu_k$,
where $\hat{\theta} \in \Lambda_M$ is any measurable function of~$Y$. Then, for any $j \in [M]$ we have the following inequality with probability at least $1-\delta$,
$$
2\langle \xi, \mu_{\hat{\theta}}-\hat \mu_j\rangle-\lambda\cK(\hat{\theta},\pi)-\sum_{k \in [M]} \hat{\theta}_k C_k
-\frac{8 \sigma^2}{\lambda}\sum_{k \in [M]} \hat \theta_k\|\hat \mu_k-\hat \mu_j\|^2\le \lambda\log\big(\frac1\delta\big)\,.
$$
Moreover,
$$
\E\Big[2\langle \xi, \mu_{\hat{\theta}}-\hat \mu_j\rangle-\lambda\cK(\hat{\theta},\pi)-\sum_{k \in [M]} \hat{\theta}_k C_k
-\frac{8 \sigma^2}{\lambda}\sum_{k \in [M]} \hat \theta_k\|\hat \mu_k-\hat \mu_j\|^2\Big]\le 0\,.
$$
\end{lemmapp}

\begin{proof}
Let $\Delta =2 \langle \xi, \mu_{\hat{\theta}}-\hat \mu_j \rangle-\lambda\cK(\hat{\theta},\pi)-\sum_{k \in [M]} \hat{\theta}_k C_k$. Then we have
\begin{align}
&\E\Big[\exp\Big(\frac{\Delta}{\lambda}-\frac{8\sigma^2}{\lambda^2}\sum_{k \in [M]} \hat \theta_k\|\hat \mu_k-\hat \mu_j\|^2\Big)\Big]\nonumber\\
= &\E\Big[\exp\Big(\sum_{k\in [M]}\hat \theta_k\big(\frac{2}{\lambda}\langle \xi, \hat \mu_k-\hat \mu_j \rangle-\log\big(\frac{\hat \theta_k}{\pi_k}\big)-\frac{C_k}{\lambda}-\frac{8\sigma^2}{\lambda^2}\|\hat \mu_k-\hat \mu_j\|^2\big)\Big)\Big]\nonumber\\
\leq &\E\Big[\sum_{k\in [M]}\hat \theta_k\exp\Big(\frac{2}{\lambda}\langle \xi, \hat \mu_k-\hat \mu_j \rangle-\log\big(\frac{\hat \theta_k}{\pi_k}\big)-\frac{C_k}{\lambda}-\frac{8\sigma^2}{\lambda^2}\|\hat \mu_k-\hat \mu_j\|^2\big)\Big)\Big]\quad \text{(Jensen's ineq.)}\nonumber\\
= &\E\Big[\sum_{k\in [M]}\pi_k\exp\Big(\frac{2}{\lambda}\langle \xi, \hat \mu_k-\hat \mu_j \rangle-\frac{C_k}{\lambda}-\frac{8\sigma^2}{\lambda^2}\|\hat \mu_k-\hat \mu_j\|^2\Big)\Big]\label{PR:T1:1}
\end{align}
Observe now that the decomposition \eqref{EQ:GMM} implies that $\hat \mu_k-\hat \mu_j=B_k\xi+v_k$ where $B_k=A_k-A_j$ and $v_k=B_k\mu+b_k-b_j$. It yields
\begin{equation}
\label{PR:T1:2}
\frac{2}{\lambda}\langle \xi, \hat \mu_k-\hat \mu_j \rangle-\frac{8\sigma^2}{\lambda^2}\|\hat \mu_k-\hat \mu_j\|^2=\xi^\top\Big[\frac{2}{\lambda} B_k -\frac{8\sigma^2}{\lambda^2}B_k^{\top} B_k\Big]\xi+ \xi^\top\Big[\frac{2}{\lambda}I_n-\frac{16\sigma^2}{\lambda^2}B_k^{\top}\Big]v_k -\frac{8\sigma^2}{\lambda^2}\|v_k\|^2\,,
\end{equation}
where $I_n$ denotes the identity matrix of $\R^n$. Next, we obtain from the Cauchy-Schwarz inequality that
$$
 \E\Big[\exp\Big(\xi^\top\Big[\frac{2}{\lambda} B_k -\frac{8\sigma^2}{\lambda^2}B_k^{\top}B_k\Big]\xi+ \xi^\top\Big[\frac{2}{\lambda}I_n-\frac{16\sigma^2}{\lambda^2}B_k^{\top}\Big]v_k\Big)\Big]\leq \sqrt{P_1\cdot P_2}\,,
$$


where,
$$
P_1=\E\Big[\exp\Big(\xi^\top\Big[\frac{4}{\lambda}B_k-\frac{16\sigma^2}{\lambda^2}B_k^{\top}B_k\Big]\xi\Big)\Big]\,,\quad
P_2=\E\Big[\exp\Big(\xi^\top\Big[\frac{4}{\lambda}I_n-\frac{32\sigma^2}{\lambda^2}B_k^{\top}\Big]v_k\Big)\Big]\,.
$$
To bound $P_1$, observe that since $A_j$ and  $B_k^{\top}B_k$ both have nonnegative eigenvalues, it holds
$$
\xi^\top\Big[\frac{4}{\lambda}B_k-\frac{16\sigma^2}{\lambda^2}B_k^{\top}B_k\Big]\xi\le \frac{4}{\lambda}\xi^\top A_k \xi=\frac{4}{\lambda}(U_k\xi)^\top D_k (V_k\xi)
$$
where $A_k=U_k^{\top}D_kV_k$ is the singular value decomposition of $A_k$. In particular, the matrices $U_k$, $V_k$ are orthogonal so that the vectors satisfy $Z=U_k\xi\sim \cN(0,\sigma^2I_n)$ and $W=V_k\xi\sim \cN(0,\sigma^2I_n)$. Since $A_k\in S_n$, we know
 $$
 \exp\Big(\frac{4}{\lambda}Z^\top D_k W\Big)\le \exp \Big(\frac{2}{\lambda}\big(Z^{\top}D_k Z+ W^{\top} D_k W \big)\Big).
 $$
 Applying now Cauchy-Schwarz inequality and Lemma~\ref{LEM:LM} yields
\begin{align}
& \E\Big[\exp\big(\frac{4}{\lambda}Z^\top D_k W\big)\Big]\le \sqrt{\E\Big[\exp \big(\frac{4}{\lambda} Z^\top D_k Z\big)\Big]}\sqrt{\E\Big[\exp \big(\frac{4}{\lambda} W^\top D_k W\big)\Big]} \cr \le & \exp\Big(\frac{4\sigma^2}{\lambda}\tr(D_k)+\frac{16\sigma^4\tr(D_k^2)}{\lambda^2-{4\sigma^2}{\lambda}\|D_k\|_{\mathsf{op}}}\Big)\le \exp\Big(\frac{4\sigma^2}{\lambda}\tr(A_k)\big(1+\frac{4\sigma^2V}{\lambda-4\sigma^2V}\big)\Big)
\end{align}

where, in the inequality, we used  the following inequalities: $\tr(D_k^2)\le \tr(D_k)\|D_k\|_{\mathsf{op}}$, $\|D_k\|_{\mathsf{op}}\le V$. Taking now $\lambda\ge 20\sigma^2V$ yields
$$
1+\frac{4\sigma^2V}{\lambda-4\sigma^2V}\le 2
$$
so that $\sqrt{P_1} \le \exp(C_k/\lambda)$, where we recall that $C_k=4\sigma^2\tr(A_k)$ is defined in~\eqref{EQ:defC_j}.

We now bound $P_2$. To that end, observe that it follows from \cite[Lemma~6.1]{Rig12} that
$$
P_2 \le \exp\Big(\frac{8\sigma^2}{\lambda^2}\Big\|\Big(I_n-\frac{8\sigma^2}{\lambda}B_k\Big)v_k\Big\|^2\Big)
$$
Note now that the eigenvalues of $B_k$ belong to $[-V, V]$ so that for $\lambda\ge 20\sigma^2V$, we have
$$
\sqrt{P_2} \le \exp\Big(\frac{196\sigma^2}{25\lambda^2} \|v_k\|^2\Big)
$$
The bounds on $\sqrt{P_1}$ and $\sqrt{P_2}$ together with~\eqref{PR:T1:1} and \eqref{PR:T1:2} yield
$$
\E\Big[\exp\Big(\frac{\Delta}{\lambda}-\frac{8\sigma^2}{\lambda^2}\sum_{k \in [M]} \hat \theta_k\|\hat \mu_k-\hat \mu_j\|^2\Big)\Big]\le 1
$$
The two statements of the lemma follow easily from this bound on the moment generating function using the same arguments as in \cite[Theorem~3.1]{Rig12}. Specifically, the statement with high probability follows from a Chernoff bound and the statement in expectation follows from the inequality $t\le e^t-1$.
\end{proof}

\subsection{Proof of Theorem~\ref{thm:T2}}

For any $\theta \in \Lambda_M$, define
\begin{align*}
\hat{S}(\theta)&=\nu\sum_{k \in [M]}\theta_k\|Y-\hat \mu_k\|^2_2+(1-\nu)\|Y-\mu_\theta\|^2_2\,,\\
S(\theta)&=\nu\sum_{k \in [M]}\theta_k\|\mu-\hat \mu_k\|^2_2+(1-\nu)\|\mu-\mu_{\theta}\|^2_2\,.
\end{align*}
and observe that
$$
\hat{S}(\theta)-S(\theta)=\|Y\|^2_2-\|\mu\|^2_2-2\langle\xi, \mu_{\theta}\rangle\,.
$$
It follows from the definition~\eqref{EQ:defQ} of $\hat{\theta}$, that for any $\theta \in \Lambda_M$, it holds
$$
\hat{S}(\hat{\theta})+\sum_{k \in [M]}\thetaphat C_k+\lambda\mathcal{K}(\hat{\theta},\pi)\leq \hat{S}(\theta)+\sum_{k \in [M]}\theta_k C_k+\lambda\mathcal{K}(\theta,\pi)\,.
$$
The above two displays yield that
\begin{equation}
\label{EQ:diffS}
S(\hat{\theta})- S(\theta)\le\sum_{k \in [M]}(\theta_k-\thetaphat) C_k+2\langle \xi, \hat \mu^Q-\mu_{\theta}\rangle+\lambda\mathcal{K}(\theta,\pi)-\lambda\mathcal{K}(\hat{\theta},\pi)\,.
\end{equation}
Observe first that
\begin{align*}
S(\hat{\theta})- S(\theta)=(1-\nu)\big[\|\mu-\hat \mu^Q\|^2-\|\mu-\mu_{\theta}\|^2\big] + \nu \sum_{k \in [M]}(\hat \theta_k-\theta_k)\|\mu-\hat \mu_k\|^2\,.
\end{align*}
Fix $\beta \in (0,1)$ and take $\theta=(1-\beta)\hat \theta +\beta e_j$ where $e_j$ denotes the $j$th vector of the canonical basis of $\R^M$. It yields
$$
\|\mu-\hat \mu^Q\|^2-\|\mu-\mu_{\theta}\|^2=\beta\big[\|\mu-\hat \mu^Q\|^2-\|\mu-\hat \mu_j\|^2\big]+\beta(1-\beta)\|\hat \mu^Q-\hat \mu_j\|^2\,.
$$
so that
\begin{align*}
\frac{1}{\beta} \big[S(\hat{\theta})- S(\theta)\big]=&(1-\nu)\big[\|\mu-\hat \mu^Q\|^2-\|\mu-\hat \mu_j\|^2\big] + (1-\nu)(1-\beta)\|\hat \mu^Q-\hat \mu_j\|^2\\
&\qquad +\nu \sum_{k \in [M]}\hat \theta_k\|\mu-\hat \mu_k\|^2-\nu \|\mu-\hat \mu_j\|^2\,.
\end{align*}
Together with the identity
\begin{equation}
\label{EQ:varid}
\sum_{k \in [M]} \hat \theta_k\|m-\hat \mu_k\|^2=\sum_{k \in [M]} \hat \theta_k\|\hat \mu^Q-\hat \mu_k\|^2+\|\hat \mu^Q-m\|^2\,,
\end{equation}
applied for $m=\hat \mu_j$  and $m=\mu$ respectively, it yields
\begin{align}
\frac{1}{\beta} \big[S(\hat{\theta})- S(\theta)\big]&=\|\mu-\hat \mu^Q\|^2-\|\mu-\hat \mu_j\|^2 + (1-\nu)(1-\beta)\sum_{k \in [M]}\hat \theta_k\|\hat \mu_j -\hat \mu_k\|^2\label{EQ:diffS2}\\
&\qquad + \big(\nu-(1-\nu)(1-\beta)\big) \sum_{k \in [M]}\hat \theta_k\|\hat \mu^Q-\hat \mu_k\|^2\,.\nonumber
\end{align}
Next, observe that,
$$
2\langle \xi, \hat \mu^Q-\mu_{\theta}\rangle=2\beta\langle \xi, \hat \mu^Q-\hat \mu_j\rangle\,,\quad \sum_{k \in [M]}(\theta_k-\thetaphat) C_k=\beta\Big[ C_j-\sum_{k \in [M]}\thetaphat C_k\Big]\,,
$$
and by convexity,
$$
\cK(\theta, \pi) \le (1-\beta)\cK(\hat \theta, \pi) + \beta \log\big(\frac1{\pi_j}\big)\,.
$$
Substituting the above expressions into~\eqref{EQ:diffS}, together with~\eqref{EQ:diffS2} yields that
\begin{align*}
\|\mu-\hat \mu^Q\|^2-&\|\mu-\hat \mu_{j}\|^2 \\
\le &\Delta + C_j +\lambda\log\big(\frac{1}{\pi_j}\big)\\
&- (1-\nu)(1-\beta)\sum_{k \in [M]}\hat \theta_k\|\hat \mu_j -\hat \mu_k\|^2- \big(\nu-(1-\nu)(1-\beta)\big) \sum_{k \in [M]}\hat \theta_k\|\hat \mu^Q-\hat \mu_k\|^2
\end{align*}
where
$$
\Delta=2\beta\langle \xi, \hat \mu^Q-\hat \mu_j\rangle-\lambda\cK(\hat \theta, \pi)-\sum_{k\in [M]}\hat \theta_kC_k
$$
as in the proof of Lemma~\ref{LEM:main}. Letting $\beta \to 0$ yields
\begin{align*}
\|\mu-\hat \mu^Q\|^2&-\|\mu-\hat \mu_{j}\|^2\\
& \le \Delta + C_j +\lambda\log\big(\frac{1}{\pi_j}\big)- (1-\nu)\sum_{k \in [M]}\hat \theta_k\|\hat \mu_j -\hat \mu_k\|^2+ (1-2\nu)\sum_{k \in [M]}\hat \theta_k\|\hat \mu^Q-\hat \mu_k\|^2\\
&\le \Delta + C_j +\lambda\log\big(\frac{1}{\pi_j}\big)- \min(\nu, 1-\nu)\sum_{k \in [M]}\hat \theta_k\|\hat \mu_j -\hat \mu_k\|^2\,,
\end{align*}
where the second inequality comes from~\eqref{EQ:varid} with $m=\hat \mu_j$ when $\nu \le 1-\nu$ (the case $\nu \ge 1-\nu$ is trivial).
It follows from Lemma~\ref{LEM:main} that
$$
\Delta\le \frac{8\sigma^2}{\lambda}\sum_{k \in [M]} \hat \theta_k\|\hat \mu_k-\hat \mu_j\|^2+ \lambda\log\big(\frac1\delta\big)
$$
with probability at least $1-\delta$ when $\lambda\ge 20 V\sigma^2$, so that taking $\lambda \ge \frac{8\sigma^2}{\min(\nu, 1-\nu,(\frac{5}{2}V)^{-1})}$, completes the proof of~\eqref{EQ:OIQP}. The proof of~\eqref{EQ:OIQE} follows by replacing the last display with the corresponding bound in expectation from Lemma~\ref{LEM:main}.

\subsection{Proof of Theorem~\ref{thm:T1}}
For any $\theta \in \Lambda_M$, define
$$
\hat{S}(\theta)=\sum_{k \in [M]}\theta_k\|Y-\hat \mu_k\|^2 ,\ \ \ S(\theta)=\sum_{k \in [M]}\theta_k\|\mu-\hat \mu_k\|^2,
$$
and observe that
\begin{equation}
\label{PR:T1:3}
\hat{S}(\theta)-S(\theta)=\|Y\|^2-\|\mu\|^2-2\langle\xi, \sum_{k \in [M]}\theta_k\hat \mu_k\rangle.
\end{equation}
It follows from the definition~\eqref{EQ:defEXP} of $\hat{\theta}$, that for any $j \in [M]$, it holds
$$
\hat{S}(\hat{\theta})+\lambda\mathcal{K}(\hat{\theta},\pi)\leq \hat{S}(e_j)+\lambda\log(\frac{1}{\pi_j})+C_j-\sum_{k \in [M]}\hat{\theta}_k C_k\,,
$$
where $e_j$ denotes the $j$th vector of the canonical basis of $\R^M$.
Together with \eqref{PR:T1:3} applied with $\theta=\hat{\theta}$ and $\theta=e_j$ respectively, and the identity
$$
\sum_{k \in [M]}\hat{\theta}_k\|\mu-\hat \mu_k\|^2=\|\mu-\hat \mu^{\mathsf{EW}}\|^2+\sum_{k \in [M]}\thetaphat\|\hat \mu_k-\hat \mu^{\mathsf{EW}}\|^2,
$$
it yields that for any $j \in [M]$, we have
\begin{equation}
\label{PR:T1:4}
\|\mu-\hat \mu^{\mathsf{EW}}\|^2\leq\|\mu-\hat \mu_j\|^2+\lambda\log\big(\frac{1}{\pi_j}\big)+C_j-\sum_{k \in [M]}\thetaphat\|\hat \mu_k-\hat \mu^{\mathsf{EW}}\|^2+\Delta,
\end{equation}
where $\Delta=2\langle\xi,\hat \mu^{\mathsf{EW}}-\hat \mu_j\rangle-\lambda\mathcal{K}(\hat{\theta},\pi)-\sum_{k \in [M]}\thetaphat C_k$.

For any $\lambda\ge 20V\sigma^2$, $\delta>0$,  Lemma~\ref{LEM:main} yields that
$$
\Delta\leq \frac{8\sigma^2}{\lambda}\sum_{k \in [M]}\thetaphat \|\hat \mu_k-\hat \mu_j\|^2+\lambda\log\big(\frac{1}{\delta\pi_j}\big)\,,
$$
with probability at least $1-\delta\pi_j$.
Together with \eqref{PR:T1:4} the identity
$$
\sum_{k \in [M]}\thetaphat\|\hat \mu_k-\hat \mu_j\|^2=\sum_{k \in [M]}\thetaphat\|\hat \mu_k-\hat \mu^{\mathsf{EW}}\|^2+\|\hat \mu^{\mathsf{EW}}-\hat \mu_j\|^2,
$$
it yields
$$
\|\mu-\hat \mu^{\mathsf{EW}}\|^2\leq\|\mu-\hat \mu_j\|^2+\frac{8\sigma^2}{\lambda}\|\hat\mu_j-\mu^{\mathsf{EW}}\|^2+\lambda\log\big(\frac{1}{\delta\pi_j^2}\big)+C_j+\big(\frac{8\sigma^2}{\lambda}-1\big)\sum_{k \in [M]}\thetaphat\|\hat \mu_k-\hat \mu^{\mathsf{EW}}\|^2
$$
Recall that our assumptions imply that $\lambda > 16\sigma^2$ so that
%
$$
\big(1-\frac{16\sigma^2}{\lambda}\big)\|\mu-\hat \mu^{\mathsf{EW}}\|^2\leq (1+\frac{16\sigma^2}{\lambda}) \|\mu-\hat\mu_j\|^2+\lambda\log\big(\frac{1}{\delta\pi_j^2}\big)+C_j.
$$
Next, observe that $(1-x)^{-1}=1+x(1-x)^{-1}\le 1+4x/3$ for $x \in (0,1/4)$ so, for $\lambda\ge 64\sigma^2$, we get 
%

$$
\|\mu-\hat \mu^{\mathsf{EW}}\|^2 \leq \big(1+\frac{128\sigma^2}{3\lambda}\big)\|\mu-\hat\mu_j\|^2+\frac{8\lambda}{3}\log\big(\frac{1}{\delta\pi_j}\big)+\frac{4 C_j}{3}.
$$
The proof is concluded by a union bound.


\subsection{Proof of Theorem~\ref{TH:spa}}

Let $\bar \beta \in \R^p$ realize the minimum in the right-hand side of~\eqref{EQ:soiQ} and let $\bar J \subset [p]$ denote the support of $\bar \beta$.
On the one hand, it follows from the Pythagorean identity that
$$
\|\hat \mu_{\bar J} -\mu\|^2=\|A_{\bar J}Y-\mu\|^2=\|A_{\bar J}\mu -\mu\|^2+\|A_{\bar J}\xi\|^2
$$
Next, since $\|A_{\bar J}\xi\|^2\sim \sigma^2 \chi^2_{\rk(\X_{\bar J})}$, if follows from Lemma~\ref{LEM:LM} together with the inequality $2\sqrt{ab}\le a+b$ valid for $a,b>0$ that with probability at least $1-\delta/2$, we have
$$
\|A_{\bar J}\xi\|^2\le 2\sigma^2\rk(\X_{\bar J})+3\sigma^2\log(2/\delta)\le 2\sigma^2|\bar \beta|_0+3\sigma^2\log(2/\delta)\,.
$$
On the other hand, we get from Theorem~\ref{thm:T2} that with probability at least $1-\delta/2$, it holds
$$
\|\hat \mu^Q-\mu\|^2 \leq \|\hat \mu_{\bar J}-\mu\|^2+C_{\bar J}+\lambda\log\big(\frac{2}{\pi_{\bar J}\delta}\big).
$$
It can be shown \cite{RigTsy12} that
$$
\log(\pi_{\bar J}^{-1})\le 2|\bar J|\log\big(\frac{e p}{|\bar J|}\big)+\frac12\le 2|\bar \beta|_0\log\big(\frac{ep}{|\bar \beta|_0}\big)+\frac12\,.
$$
and we also have that $C_{\bar J}=4\sigma^2\rk(\X_{\bar J})\le 4\sigma^2|\bar \beta|_0$.

Putting everything together yields that with probability at least $1-\delta$, it holds
\begin{align*}
\|\hat \mu^Q-\mu\|^2 &\leq \|A_{\bar J} \mu -\mu\|^2 + 6\sigma^2|\bar \beta|_0+ 2\lambda|\bar \beta|_0\log\big(\frac{ep}{|\bar \beta|_0}\big)+ \frac{\lambda}{2} + (3\sigma^2+\lambda) \log(2/\delta) \\
&\le \|A_{\bar J} \mu -\mu\|^2 + \big(\frac{5\lambda}{2}+6\sigma^2\big)|\bar \beta|_0\log\big(\frac{2ep}{|\bar \beta|_0\delta}\big)\,.
\end{align*}
To conclude the proof of~\eqref{EQ:soiQ}, it suffices to observe that $ \|A_{\bar J} \mu -\mu\|^2 \le  \|\X\bar \beta -\mu\|^2$.

The proof of~\eqref{EQ:soiEXP} follows  along the same lines.

\subsection{Proof of Theorem~\ref{TH:univagg}}

Replacing $\beta$ by $\theta$ and $\X_j$ by $\mu_j$ in the proof of Theorem~\ref{TH:spa} leads to
$$
\|\hat \mu^Q-\mu\|^2 \le \min_{\theta \in \R^M}\big\{\|\mu_\theta-\mu\|^2+  66 \sigma^2|\theta|_0\log\big(\frac{2eM}{|\theta|_0\delta}\big)\big\}\,.
$$
The above display combined with Lemma~\ref{LEM:maurey2} yields that for any $q \in (0,1), R>0$,
$$
\|\hat \mu^Q-\mu\|^2 \le \min_{\theta \in \R^M}\big\{\|\mu_\theta-\mu\|^2+ 66 \sigma^2|\theta|_0\log\big(\frac{2eM}{|\theta|_0\delta}\big)\wedge \varphi_{q,M}(\theta; 9\sigma, B)\big\}\,,
$$
where the function $\varphi_{q,M}$ is defined in~\eqref{EQ:defPHI}. To complete the proof, if suffices that for any $\Theta \in \{\sB_0(1)\cap\sB_1(1),\sB_1(1), \R^M, \sB_0(D), \sB_0(D)\cap \sB_1(1), \sB_q(R), \sB_0(D)\cap\sB_q(R)\}$, there exists $q \in (0,1)$ and $R>0$ such that
$$
\sup_{\theta \in \Theta} \big\{66 \sigma^2|\theta|_0\log\big(\frac{2eM}{|\theta|_0\delta}\big)\wedge \varphi_{q,M}(\theta; 9\sigma, B)\big\}\le C\Delta_{n,M}(\Theta)\,.
$$
In the rest of the proof, we treat each case separately. To that ends, write
$$
\psi(\theta)=66 \sigma^2|\theta|_0\log\big(\frac{2eM}{|\theta|_0\delta}\big)\wedge \varphi_{q,M}(\theta; 9\sigma, B)\,.
$$

{\sc Model Selection aggregation.} If $\Theta=\sB_0(1)$, observe that for any $\theta \in \Theta$,
$$
\psi(\theta)\le 66 \sigma^2|\theta|_0\log\big(\frac{2eM}{|\theta|_0\delta}\big)\le 66\sigma^2\log(\frac{2eM}{\delta})\,.
$$

{\sc $\ell_q$ and convex aggregation.} If $\Theta=\sB_q(R)$ (and in particular, $\Theta=\sB_1(1)$ for convex aggregation), observe that for any $\theta \in \Theta$,
\begin{align*}
\psi(\theta)&\le 66\sigma^2M\log(2e/\delta)\wedge\varphi_{q,M}(\theta; 9\sigma, B)\\
&\le \Big[17(9\sigma)^{2-q}R^qB^q \Big[\blog\Big(\frac{eM}{\delta}\big(\frac{9\sigma}{BR}\big)^q\Big)\Big]^{1-\frac{q}{2}}\vee 198\sigma^2 \blog\big(\frac{eM}{\delta}\big)\,\Big]\wedge 66\sigma^2 M\log\big(\frac{2e}{\delta}\big) \,.
\end{align*}

{\sc Linear aggregation.} If $\Theta=\R^M$, observe that for any $\theta \in \Theta$,
\begin{align*}
\psi(\theta)\le 66 \sigma^2|\theta|_0\log\big(\frac{2eM}{|\theta|_0\delta}\big)\le 66\sigma^2 M\log\big(\frac{2e}{\delta}\big)\,.
\end{align*}

{\sc $D$-linear aggregation.} If $\Theta=\sB_0(D)$, observe that for any $\theta \in \Theta$,
\begin{align*}
\psi(\theta)\le 66 \sigma^2|\theta|_0\log\big(\frac{2eM}{|\theta|_0\delta}\big)\le 66 \sigma^2D\log\big(\frac{2eM}{D\delta}\big)\,.
\end{align*}

{\sc $D$-convex aggregation.} If $\Theta=\sB_0(D)\cap\sB_1(1)$, observe that for any $\theta \in \Theta$,
\begin{align*}
\psi(\theta)&\le 66 \sigma^2|\theta|_0\log\big(\frac{2eM}{|\theta|_0\delta}\big)\wedge\varphi_{1,M}(\theta; 9\sigma, B)\\
&\le \Big[153\sigma B \Big[\blog\Big(\frac{9eM\sigma}{\delta B}\Big)\Big]^{\frac{1}{2}} \vee 198\sigma^2 \blog\big(\frac{eM}{\delta}\big)\, \Big]\wedge 66 \sigma^2D\log\big(\frac{2eM}{D\delta}\big).
\end{align*}

{\sc $D$-$\ell_q$ and $D$-convex aggregation.} If $\Theta=\sB_0(D)\cap\sB_q(R)$  (and in particular, $q=1, R=1$ for convex aggregation), observe that for any $\theta \in \Theta$,
\begin{align*}
\psi(\theta)&\le 66 \sigma^2|\theta|_0\log\big(\frac{2eM}{|\theta|_0\delta}\big)\wedge\varphi_{q,M}(\theta; 9\sigma, B)\\
& \le \Big[ 17(9\sigma)^{2-q}R^qB^q \Big[\blog\Big(\frac{eM}{\delta}\big(\frac{9\sigma}{BR}\big)^q\Big)\Big]^{1-\frac{q}{2}}\vee 198\sigma^2 \blog\big(\frac{eM}{\delta}\big)\Big]\wedge 66\sigma^2 D\log\big(\frac{2eM}{D\delta}\big).
\end{align*}

\section{A generalized Maurey argument}
\label{SEC:maurey}
\subsection{Decay of coefficients on $\ell_q$-balls}
For any $q>0, \theta \in \R^M$, recall that $|\theta|_q$ denotes the $\ell_q$-norm of $\theta$ and is defined by
$$
|\theta|_q=\Big( \sum_{j \in [M]} |\theta_j|^q\Big)^\frac1q\,.
$$
It is known \cite{Joh11} that  if $q<1$, such balls contain sparse signals, in the sense that their coefficients decay at a certain polynomial rate. This is quantified by the following lemma that yields a much sharper result than the one obtained using weak $\ell_q$-balls, especially for $q$ close to $1$.
\begin{lemmapp}
\label{LEM:decay}
Fix $R>0$ and $q \in (0,1)$. For any $\theta \in \sB_q(R)$, let $|\theta_{(1)}|\ge \ldots \ge |\theta_{(M)}|$ denote a non-increasing rearrangement of the absolute values of the coefficients of $\theta$. Then for any integer $m$ such that $1\le m \le M$, it holds
$$
\sum_{j=m+1}^M |\theta_{(j)}| \le |\theta|_qm^{1-\frac1q}\,.
$$
\end{lemmapp}
\begin{proof}
Let $\{v_j\}_{j\ge 1}$ be an infinite sequence such that $v_j=|\theta_{(j)}|$ for $j \in [M]$ and  $v_j=0$ for $j \ge M+1$.
Next for any $k\ge 0$, let $B_k$ denote the block of $m$ consecutive integers defined by $B_k=\{km+1, \ldots, (k+1)m\}$ and observe that
\begin{align*}
\sum_{j=m+1}^M | \theta_{(j)}| =\sum_{j\ge m+1}v_j&=\sum_{k\ge 1}\sum_{j \in B_k}v_j=\sum_{k\ge 1}\sum_{j \in B_k}(v_j^q)^{\frac1q}\\
&\le \sum_{k\ge 1}\sum_{j \in B_k}\Big(\frac{1}{m}\sum_{i \in B_{k-1}}v_i^q\Big)^{\frac1q}\\
&=m^{1-\frac1q}\sum_{k\ge 1}\Big(\sum_{i \in B_{k-1}}v_i^q\Big)^{\frac1q}\\
&\le m^{1-\frac1q}\Big(\sum_{k\ge 1}\sum_{i \in B_{k-1}}v_i^q\Big)^{\frac1q}\\
&=|\theta|_q m^{1-{\frac1q}}\,,
\end{align*}
where in the last inequality, we use the fact that $a^p + b^p \le (a+b)^p$ for any $a,b>0, p\ge 1$.
\end{proof}

\subsection{Proof of Lemma~\ref{LEM:maurey2}}

We begin by an approximation bound {\it a la} Maurey on $\ell_q$ balls.
\begin{lemmapp}
\label{LEM:maurey}
Let $\mu_1, \ldots, \mu_M \in \R^M$ be such that $\max_j\|\mu_j\|^2 \le B^2$. Then for any $\mu \in \R^M$, any $q,\theta$, and any positive integer $m \le M/2$, there exists $\theta^m \in \R^M$ such that $|\theta^m|_0 \le 2m$ and
\begin{equation}
\label{EQ:maurey}
\|\mu_{\theta^m}-\mu\|^2 \le \|\mu_{\theta}-\mu\|^2 + B^2|\theta|_q^2m^{1-\frac2q}\,.
\end{equation}
\end{lemmapp}
\begin{proof}
Fix $q \in (0,1]$ and $\theta \in \R^M$. Denote by $|\theta_{(1)}|\geq\ldots\geq |\theta_{(M)}|\geq0$ a non-decreasing rearrangement of the absolute value of the coordinates of $\theta$.  Next, decompose the vector $\theta$ into $\theta=\alpha + \beta$ so that $\mu_{\theta}=\mu_\alpha +\mu_\beta$, where $\alpha$ and $\beta$ have disjoint support and $\alpha \in \sB_0(m)$ is supported by the $m$ indices with the largest absolute coordinates of $\theta$.  Since $\theta\in \sB_q$, it follows form Lemma~\ref{LEM:decay} that the $\ell_1$-norm of $\beta=\theta-\alpha$ satisfies
$$
|\beta|_1= \sum_{j=m+1}^M|\theta_{(j)}|\leq |\theta|_q m^{1-\frac1q} =:r\,.
$$
Therefore, $\beta\in r\sB_1$. We now use Maurey's empirical method \cite{Pis81} to find a $m$-sparse approximate of $\mu_\beta$. Define a random vector $U \in \R^M$ with values in $\{0,\pm r\mu_1,\ldots,\pm r\mu_M \}$ by $P[U= r{\rm sign}(\beta_i)\mu_i]=|\beta_i|/r$ and $P[U=0]=1-|\beta|_1/r$. Let $U_1, \ldots, U_m$ be \iid  copies of $U$ and notice that $\E[U]=\mu_\beta$ and $\|U\|\leq r\max_{j}\|\mu_j\|\le r B$. It yields,
 \begin{align*}
   \E\|\mu-\mu_\alpha-\frac{1}{m}\sum_{i=1}^mU_i\|^2=\|\mu-\mu_{\theta}\|^2+\frac{\E\|U-\E U\|^2}{m}\leq \|\mu-\mu_{\theta}\|^2+\frac{(rB)^2}{m}.
 \end{align*}
 Therefore there exists some realization $\mu_{\theta^m}$ of the random vector $\mu_\alpha+\frac{1}{m}\sum_{i=1}^mU_i$  for which \eqref{EQ:maurey} holds and $|\theta^m|_0\leq 2m$.
\end{proof}

We now return to the proof of Lemma~\ref{LEM:maurey2}. Define
$$
A=\min_{\theta \in \R^M} \Big\{ \|\mu_\theta -\mu\|^2 +\nu^2 |\theta|_0\log\big(\frac{2eM}{|\theta|_0\delta}\big)\Big\}\,.
$$
Fix $\theta \in \R^M$ and  define $m =\lceil x\rceil$ where
$$
 x = \frac{B^q| \theta|^q_q}{\nu^q}[\blog\big(\frac{eM\nu^q}{B^q |\theta|^q_q \delta}\big)]^{-\frac{q}{2}}> 0,
$$
First, if $ m > | \theta|_0/2$,  we use the simple bound
$$
A\le \|\mu_{ \theta} -\mu\|^2 +\nu^2 |\theta|_0\log\big(\frac{2eM}{|\theta|_0\delta}\big)\ \le \|\mu_{ \theta} -\mu\|^2+ 2\nu^2  m \blog \big(\frac{eM}{ m \delta}\big).
$$
Next, if $ m \le | \theta|_0/2$, it follows from Lemma~\ref{LEM:maurey} that  there exists $\theta^{ m}$ such that $|\theta^{ m}|_0 \le 2{ m}$ and
\begin{align*}
A&\le  \|\mu_{\theta^m} -\mu\|^2 +2\nu^2 m\log\big(\frac{eM}{m\delta}\big)\\
&\le  \|\mu_{ \theta}-\mu\|^2+ 2\nu^2 m\log\big(\frac{eM}{m\delta}\big)+B^2| \theta|_q^2m^{1-\frac2q}\,.
\end{align*}
Therefore, whether $ m > | \theta|_0/2$ or $ m \le | \theta|_0/2$, it holds for any $\theta \in \R^M$,
\begin{equation}
\label{EQ:boundA}
A\le  \|\mu_{ \theta}-\mu\|^2+ 2\nu^2 m\log\big(\frac{eM}{m\delta}\big)+B^2| \theta|_q^2m^{1-\frac2q}\,.
\end{equation}
To control the right-hand side of~\eqref{EQ:boundA}, consider two cases for the value of $x$.\\
\noindent{\sc Case 1:} If $ x < 1$, we have $m =1$ and we will show $B^2| \theta|^2_q\leq \nu^2\blog(eM/\delta)$. Indeed,
if $B|\theta|_q\le \nu$, then this bound holds trivially and if $B|\theta|_q\ge  \nu$, then $x \ge (B|\theta|_q/\nu)^q[\blog(eM/\delta)]^{-\frac{q}{2}}$. Together with $x<1$, the last inequality implies that $B^2| \theta|^2_q\leq \nu^2\blog(eM/\delta)$.
%
%
%
Therefore, in {\sc Case 1}, we have
$$
A\le \|\mu_{ \theta} -\mu\|^2+3\nu^2(eM/\delta).
$$
\noindent{\sc Case 2:}If $ x\geq 1$, then $ x\leq  m\leq 2 x$. Together with  the fact that $\blog\big(t\blog(t)\big)\leq 2\blog(t)$, for any $t>0$, it yields
\begin{align*}
2\nu^2 m \blog \big(\frac{eM}{ m \delta}\big)+B^2 &| \theta|^2_q { m}^{1-\frac{2}{q}}\leq 4\nu^2 x \blog\big(\frac{eM}{2  x \delta}\big)+B^2 | \theta|^2_q { x}^{1-\frac{2}{q}}\cr
&\leq 16 \nu^{2-q}B^q |\theta|^q_q \Big[\blog\big(\frac{eM\nu^q}{B^q|\theta|^q_q\delta}\big)\Big]^{1-\frac{q}{2}}+ \nu^{2-q}B^q |\theta|^q_q\Big[\blog\big(\frac{eM\nu^q}{B^q |\theta|^q_q\delta}\big)\Big]^{1-\frac{q}{2}}\cr
&\leq 17\nu^{2-q}B^q |\theta|^q_q \Big
[\blog\big(\frac{eM\nu^q}{B^q|\theta|^q_q \delta}\big)\Big]^{1-\frac{q}{2}}\,.
\end{align*}

Putting the two cases together with~\eqref{EQ:boundA}, we get that
\begin{align*}
A\le
\min_{\theta \in \R^M}  \Big\{ \|\mu_\theta -\mu\|^2 +\varphi_{q,M}(\theta; \nu, B)\,\Big\}\,,
\end{align*}
where
$$
\varphi_{q,M}(\theta; \nu, B)\,=3\nu^2\blog\big(\frac{eM}{\delta}\big)
 \vee 17\nu^{2-q}B^q |\theta|^q_q \Big
[\blog\big(\frac{eM\nu^q}{B^q|\theta|^q_q \delta}\big)\Big]^{1-\frac{q}{2}}.
$$

\section{Technical lemmas}

\subsection{Deviations of a $\chi^2$ distribution}
Let us first recall Lemma~1 of~\cite{LauMas00} in a form that is adapted to our purpose. We omit its proof.
\begin{lemmapp}\label{LEM:LM}
Suppose $(Z_1,\cdots, Z_k)$ are i.i.d. standard Gaussian random variables. Let $a_1,\cdots,a_k$ be nonnegative numbers and define
$|a|_{\infty}=\max_{i \in [k]}a_i$, $|a|^2_2=\sum^k_{i=1}a^2_i$.
Let
\[
S=\sum^{k}_{i=1}a_i(Z^2_i-1).
\]
Then for any $u$ such that $0<2|a|_{\infty}u<1$, it holds
\[
\E\left[\exp\left(uS\right)\right]\leq \exp\Big(\frac{|a|^2_2 u^2}{1-2|a|_{\infty}u}\Big).
\]
and for any $t>0$,
$$
\p(S>2|a|_2\sqrt{t}+2|a|_\infty t)\le e^{-t}\,.
$$
\end{lemmapp}

\bibliographystyle{amsalpha}
\bibliography{DaiRigXia13}

\end{document}

%% file: commandes.tex
\newcommand{\cF}{\mathcal{F}}

\newcommand{\cH}{\mathcal{H}}

\newcommand{\cK}{\mathcal{K}}

\newcommand{\cN}{\mathcal{N}}

\newcommand{\cS}{\mathcal{S}}

\newcommand{\R}{{\rm I}\kern-0.18em{\rm R}}
\newcommand{\h}{{\rm I}\kern-0.18em{\rm H}}
\newcommand{\K}{{\rm I}\kern-0.18em{\rm K}}
\newcommand{\p}{{\rm I}\kern-0.18em{\rm P}}
\newcommand{\E}{{\rm I}\kern-0.18em{\rm E}}
\newcommand{\Z}{{\rm Z}\kern-0.18em{\rm Z}}
\newcommand{\1}{{\rm 1}\kern-0.24em{\rm I}}
\newcommand{\N}{{\rm I}\kern-0.18em{\rm N}}

\newcommand{\field}[1]{\mathbb{#1}}
\newcommand{\X}{\field{X}}

\newcommand{\pp}{{\sf p}}

\newcommand{\pn}{\p_{\kern-0.25em n}}
\newcommand{\pnm}{\p_{\kern-0.25em n,m}}
\newcommand{\psubm}{\p_{\kern-0.25em m}}
\newcommand{\psubp}{\p_{\kern-0.25em p}}
\newcommand{\cfi}{\cF_{\kern-0.25em \infty}}

\newcommand{\argmin}{\mathop{\mathrm{argmin}}}

\newcommand{\eps}{\varepsilon}

\newlength{\minipagewidth}